\documentclass[11pt]{article}
\usepackage{amsfonts,mathtools}
\usepackage{amsmath,bbm}
\usepackage[usenames,dvipsnames]{color}
\usepackage[textwidth=2.3cm]{todonotes}

\newcommand{\beq}{\begin{equation}}
\newcommand{\eeq}{\end{equation}}

\newcommand{\bea}{\begin{eqnarray}}
\newcommand{\eea}{\end{eqnarray}}
\newcommand{\beas}{\begin{eqnarray*}}
\newcommand{\eeas}{\end{eqnarray*}}

\newtheorem{theorem}{Theorem}[section]
\newtheorem{teor}[theorem]{Theorem}

\newtheorem{definition}[theorem]{Definition}

\newtheorem{corollary}[theorem]{Corollary}
\newtheorem{lemma}[theorem]{Lemma}
\newtheorem{remark}[theorem]{Remark}
\newtheorem{example}[theorem]{Example}
\newtheorem{examples}[theorem]{Examples}
\newtheorem{foo}[theorem]{Remarks}

\newenvironment{proof}{\addvspace{\medskipamount}\par\noindent{\it
Proof}.}
{\unskip\nobreak\hfill$\Box$\par\addvspace{\medskipamount}}

\newcommand{\R}{\mathbb R}

\newcommand{\p}{\mathbb P}

\newcommand{\bS}{\mathbb S}
\newcommand{\bH}{\mathbb H}

\newcommand{\HP}{\mathbb HP}
\newcommand{\HH}{\mathbb HH}
\newcommand{\E}{\mathbb E}
\DeclareMathOperator{\sech}{sech}

\title{Quaternionic Brownian windings}
\author{Fabrice Baudoin\footnote{Author supported in part by the Simons Foundation and NSF Grant  DMS-1901315.}, Nizar Demni, Jing Wang }
\begin{document}
\maketitle


\begin{abstract}
We define and study the 3-dimensional windings along Brownian paths in the quaternionic Euclidean, projective and hyperbolic spaces. In particular, the asymptotic laws of these windings are shown to be Gaussian for the flat and spherical geometries while the hyperbolic winding exhibits a different long time-behavior. The corresponding asymptotic law seems to be new and is related to the Cauchy relativistic distribution.  
\end{abstract}

\tableofcontents

\section{Introduction}

In the punctured complex plane $\mathbb C \setminus \{ 0 \}$, consider the one-form
\[
\alpha=\frac{x dy-ydx}{x^2+y^2}.
\]
For every smooth path $\gamma: [0,+\infty) \to \mathbb{C} \setminus \{ 0 \}$ one has the representation
\[
\gamma(t)=  | \gamma (t) | \exp \left( i \int_{\gamma[0,t]} \alpha \right), \quad t \ge 0.
\]
It is therefore natural to call $\alpha$ the winding form around 0 since the integral of a smooth path $\gamma$ along this form quantifies the angular motion of this path. The integral of the winding form along the paths of a two-dimensional Brownian motion  $(B (t))_{t \ge 0}$ which is not started from 0 can be defined using It\^o's calculus and yields the Brownian winding functional:
\[
\zeta (t)=\int_{B[0,t]} \alpha.
\]
This functional and several natural variations of it have been extensively studied in the literature. We refer the reader to \cite{Yor1}, \cite{RY}, \cite{BW} and references therein for more details.

Our goal in this paper is to introduce a natural generalization of the winding form in homogeneous 4-dimensional spaces equipped with a quaternionic structure and study the limiting laws of the integrals of this form along the corresponding Brownian motion paths. It turns out that such one-form is  valued in the three-dimensional Lie algebra $\mathfrak{su}(2)$ and quantifies in a natural way the angular motion of a path. Actually, it may be defined by taking advantage of the fact that in a four-dimensional homogeneous and quaternionic manifold, unit spheres are canonically isometric to the Lie group $\mathbf{SU}(2)$ and thus the Lie group structure of spheres allows to consider the logarithm of a path in the sense of Chen \cite{Chen54}. In this respect, the winding form integrated along a path $\gamma$ is equal  to the Maurer-Cartan form of $\mathbf{SU}(2)$ integrated along the spherical part $\frac{\gamma}{| \gamma |}$ of this path. 

The classification of 4-dimensional homogeneous and quaternionic manifolds is well-known, and up to equivalence there are only three such spaces: the field of quaternions $\mathbb H$, the quaternionic projective line $\mathbb{H}P^1$ and the quaternionic hyperbolic space $\mathbb{H}H^1$. Our main results are the following: Let $\zeta (t)$ be the quaternionic Brownian winding functional, then:
\begin{itemize}
\item On $\mathbb{H}$, the following convergence in distribution holds:
\begin{equation*}
\lim_{t \rightarrow +\infty}\frac{2}{\sqrt{\log t}} \zeta (t) = \mathcal{N}(0, {\bf I_3})
\end{equation*}
where ${\bf I_3}$ is the $3 \times 3$ identity matrix. 
\item On $\mathbb{H}P^1$, when $t \to \infty$, in distribution we have
\[
\lim_{t \to +\infty} \frac{\zeta (t)}{t}  =\mathcal{N}(0, 2{\bf I_3}) .
\]
\item On $\mathbb{H}H^1$, the following convergence in distribution holds:
\begin{equation*}
\lim_{t \rightarrow +\infty} \zeta (t) = C
\end{equation*}
where $C$ is a 3-dimensional random variable such that
\[
\mathbb{E}(e^{i \lambda \cdot C})=\tanh (r (0))^{\sqrt{|\lambda|^2+1}-1} \left( 1 +\frac{1}{2\cosh(r (0))^2}(\sqrt{|\lambda|^2+1}-1) \right)
\]
and $r (0)$ is the initial distance from the origin in $\mathbb{H}H^1$ to the starting point of the Brownian motion. This random variable is closely related to the three-dimensional relativistic Cauchy distribution (\cite{BMR09}) and its density is expressed below through the modified Bessel function of the second kind. 
\end{itemize}
The paper is organized as follows. The next section is concerned with the winding number in the quaternionic field. In particular, we write two proofs of the corresponding limiting result. In section 3, we define and determine the limiting behavior of the winding number in the quaternionic projective line. In the last section, we deal with the hyperbolic winding process for which we determine the limiting distribution and give an explicit expression of its density. 
\section{Winding of the quaternionic Brownian motion}

\subsection{Quaternionic winding form}

Let $\mathbb{H}$ be the quaternionic field 
\[
\mathbb{H}=\{q=t+xI+yJ+zK, (t,x,y,z)\in\R^4\},
\]
where  $I,J,K \in \mathbf{SU}(2)$ are the Pauli matrices:
\[
I=\left(
\begin{array}{ll}
i & 0 \\
0&-i 
\end{array}
\right), \quad 
J= \left(
\begin{array}{ll}
0 & 1 \\
-1 &0 
\end{array}
\right), \quad 
K= \left(
\begin{array}{ll}
0 & i \\
i &0 
\end{array}
\right).
\]
Then the quaternionic norm is given by $|q |^2 =t^2 +x^2+y^2+z^2$ and the set of unit quaternions is identified with $\mathbf{SU}(2)$. Now, consider $\gamma: [0,+\infty) \to \mathbb{H} \setminus \{ 0 \}$ is a $C^1$-path and write its polar decomposition: 
\[
\gamma (t) = | \gamma (t) | \Theta (t), \quad t \geq 0,
\]
with $\Theta (t) \in \mathbf{SU}(2)$. Then,

\begin{definition}
The winding path $(\theta(t))_{t \geq 0} \in \mathfrak{su}(2)$ along $\gamma$ is defined by:
\[
\theta(t) = \int_0^t \Theta (s)^{-1}  d\Theta (s).
\]
The quaternionic winding form is the $\mathfrak{su}(2)$-valued one-form $\eta$ such that
\[
\theta(t)=\int_{\gamma[0,t]} \eta.
\]
Equivalently, 
\begin{equation*}
\theta(t)= \int_{\Theta [0,t]} \omega, 
\end{equation*}
where $\omega$ is the Maurer-Cartan form in $\mathfrak{su}(2)$. 
\end{definition}
In order to study the stochastic winding in $\mathbb{H}$, we need to compute $\eta$ in real coordinates $(t,x,y,z)$. To this end, we write 
\begin{equation*}
\Theta = \left(\begin{array}{ll}
\theta_1 & \theta_2 \\
-\overline{\theta_2} & \overline{\theta_1}
\end{array}\right),
\end{equation*}
where 
\begin{equation*}
\theta_1 = \frac{\gamma_0+i\gamma_1}{|\gamma|}, \quad \theta_2 = \frac{\gamma_2+i\gamma_3}{|\gamma|}, \quad \gamma = (\gamma_i)_{i=0}^3. 
\end{equation*}
Since $\Theta$ is unitary and has determinant one, then 
\begin{equation*}
\Theta^{-1} = \left(\begin{array}{ll}
\overline{\theta_1} &  -\theta_2 \\
\overline{\theta_2} & \theta_1
\end{array}\right),
\end{equation*}
so that 
\begin{equation*}
\Theta^{-1}\overset{\cdot}{\Theta} = \left(\begin{array}{ll}
\overline{\theta_1}\overset{\cdot}{\theta_1} + \theta_2 \overline{\overset{\cdot}{\theta_2}}  &   \overline{\theta_1}\overset{\cdot}{\theta_2} - \theta_2 \overline{\overset{\cdot}{\theta_1}} \\
 \overline{\theta_2} \overset{\cdot}{\theta_1} - \theta_1 \overline{\overset{\cdot}{\theta_2}}  &  \overline{\theta_2}\overset{\cdot}{\theta_2} + \theta_1 \overline{\overset{\cdot}{\theta_1}}  
\end{array}\right).
\end{equation*}
After straightforward computations, we end up with the following expression of $\eta = \eta_1 I + \eta_2  J + \eta_3 K$: 
\begin{eqnarray*}
\eta_1 & = & \frac{tdx- xdt + zdy - ydz}{|q|^2} \\ 
\eta_2 &= & \frac{tdy- ydt + xdz - zdx}{|q|^2} \\ 
\eta_3 & = & \frac{tdz- zdt + ydx - xdy}{|q|^2} .
\end{eqnarray*} 

Note that $\eta$ may be more concisely written in quaternionic coordinates as: 
\begin{equation*}
\eta = \frac{1}{2} \left(\frac{\overline{q}dq - \overline{dq}q}{|q|^2}\right) = \frac{1}{|q|^2}\mathrm{Im}(\overline{q}dq),
\end{equation*}
where 
\begin{equation*}
\overline{q} = t - xI - yJ - z K, \quad dq = dt + dx I + dy J + dz K, 
\end{equation*}
are respectively the quaternonic conjugate and differential.

\subsection{Asymptotic winding of the quaternionic Brownian motion}
From the previous paragraph, we are led to the following definition: 
\begin{definition}
The winding number of a quaternionic Brownian motion $W=W_0+W_1 I+W_2J + W_3K $, not started from 0, is defined by the Stratonovitch stochastic line integral:
\begin{equation*}
\zeta (t) := \int_{W(0,t]} \eta, \quad t \geq 0.
\end{equation*}
\end{definition}
The study $\zeta$ is based on the following lemma which is a well-known consequence of the skew-product decomposition of Euclidean Brownian motions (see \cite{PR}).

\begin{lemma}
Let  $W=W_0+W_1 I+W_2J + W_3K $ be a quaternionic Brownian motion not started from 0. There exist a 
 Bessel process $(R(t))_{t \ge 0}$ of dimension four (or equivalently index one) and a $\mathbf{SU}(2)$-valued Brownian motion $(\Theta (t))_{t \geq 0}$ independent from the process  $(R(t))_{t \ge 0}$ such that 
 \[
 W(t)=R(t) \Theta (A_t),
 \]
 where 
 \begin{equation*}
A_t:= \int_0^t \frac{ds}{R^2 (s)}.
\end{equation*}
\end{lemma}
As a consequence of the previous lemma, one readily has:
\[
\zeta (t)= \int_0^t \Theta (A_s)^{-1}  \circ d\Theta (A_s)=\int_0^{A_t} \Theta (s)^{-1}  \circ d\Theta (s).
\]
Since
\[
B (t) := \int_0^{t} \Theta (s)^{-1}  \circ d\Theta (s), \quad t \geq 0,
\]
is a three-dimensional Euclidean Brownian motion, we deduce that the quaternionic winding process $\zeta$ has the same distribution as: 
\begin{equation*}
(B^1 (A_t), B^2 (A_t), B^3 (A_t))_{t \geq 0}.
\end{equation*}

As a result, the characteristic function of the winding process at time $t$ is given by: 
\begin{equation*}
\mathbb{E}_{\rho}[e^{i\lambda \cdot \zeta (t)}] = \mathbb{E}_{\rho}[e^{-|\lambda|^2 A_t/2}], \quad \rho := |W_0| > 0,
\end{equation*}
for any $\lambda \in \mathbb{R}^3$. But according to the Hartman-Watson law (see \cite{Yor1}), one has 
\begin{equation*}
\mathbb{E}_{\rho}[e^{-|\lambda|^2 A_t/2} | R(t) = r] = \frac{I_{\sqrt{1+|\lambda|^2}}(r\rho/t)}{I_{1}(r\rho/t)},
\end{equation*}
where $I_{\nu}$ stands for the modified Bessel function. Appealing further to the semigroup density of the Bessel process (\cite{RY}), it follows that:
\begin{align}\label{characteristic winding}
\mathbb{E}_{\rho}[e^{i\lambda \cdot \zeta (t)}] &= \frac{e^{-\rho^2/(2t)}}{t\rho} \int_0^{\infty} I_{\sqrt{1+|\lambda|^2}}\left(\frac{r\rho}{t}\right) e^{-r^2/(2t)} r^2 dr. 
\end{align}
Using this integral representation, we are now able to determine the limiting behavior of $\zeta (t)$ as $t \rightarrow \infty$.

\begin{teor}\label{T1}
The following convergence in distribution holds:
\begin{equation*}
\lim_{t \rightarrow +\infty}\frac{2}{\sqrt{\log t}} \zeta (t) = \mathcal{N}(0, {\bf I_3})
\end{equation*}
where ${\bf I_3}$ is the $3 \times 3$ identity matrix. 
\end{teor}

\begin{proof}
We shall give two proofs of this limit theorem. The first proof relies on the explicit representation \eqref{characteristic winding} and the second one on Girsanov's theorem. The second proof is easier to generalize to the curved geometric settings studied afterwards.

\underline{\textbf{Proof 1}}: Performing the variable change $r \mapsto \sqrt{r}$ in the  integral \eqref{characteristic winding}, we get 
\begin{equation*}
\mathbb{E}_{\rho}[e^{i\lambda \cdot \zeta (t)}] = \frac{e^{-\rho^2/(2t)}}{\rho} \int_0^{\infty} \sqrt{t} I_{\sqrt{1+|\lambda|^2}}\left(\frac{r\rho}{\sqrt{t}}\right) e^{-r^2/2} r^2 dr. 
\end{equation*}
Expanding further the Bessel function: 
\begin{equation*}
I_{\sqrt{1+|\lambda|^2}}\left(\frac{r\rho}{\sqrt{t}}\right) = \sum_{j \geq 0}\frac{1}{\Gamma(j+1+\sqrt{1+|\lambda|^2})j!} \left(\frac{r\rho}{\sqrt{t}}\right)^{2j+\sqrt{1+|\lambda|^2}},
\end{equation*}
we infer that the large-time behavior of $\zeta (t)$ is governed by the lowest-order term. Finally, rescaling $\lambda$ by $\sqrt{\log t}/\sqrt{2}$, then 
\begin{equation*}
\lim_{t \rightarrow +\infty} e^{-(\log(t)/2)(\sqrt{1+2|\lambda|^2/\log(t)} - 1)} = e^{-|\lambda|^2/2}
\end{equation*}
whence the result follows. 

\underline{\textbf{Proof 2}}: We can also derive the limiting behavior of $\zeta$ by using Girsanov's theorem as well. To proceed, recall the stochastic differential equation satisfied by the Bessel process $R$: 
\[
dR(t)=\frac3{2R(t)}dt+d\xi_t, \quad R (0) = \rho > 0,
\]
where $(\xi_t)_{t \geq 0}$ is a one-dimensional standard Brownian motion. Then, letting $\mu=\sqrt{|\lambda|^2+1}-1$, we can consider the martingale
\[
D_t^{(\mu)}=\exp\bigg(\mu\int_0^t\frac{1}{R (s)}d\xi_s-\frac{\mu^2}{2}\int_0^t\frac{1}{R (s)^2}ds \bigg).
\]
By It\^o's formula, we have
\[
D_t^{(\mu)}=\left(\frac{R(t)}{\rho}\right)^\mu\, \exp\left(-\left(\frac{1}{2}\mu^2+\mu\right)\int_0^t\frac{1}{R (s)^2}ds \right)
\]
Hence, Girsanov's theorem shows that $(R(t))_{t \geq 0}$ is  a Bessel process of dimension $4\mu+3$ under the probability measure $\mathbb{P}^{(\mu)}$ with Radon-Nikodym density $D_t^{(\mu)}$ and 
\[
\mathbb{E}_{\rho}[e^{i\lambda \cdot \zeta (t)}] = \E_{\rho}\left(e^{-\frac{|\lambda|^2}{2}A_t} \right)=(\rho)^\mu\,\E^{(\mu)}_{\rho}\left(\frac{1}{(R(t))^{\mu}}\right). 
\]
Setting 
\begin{equation*}
\lambda_t:=\frac{\sqrt{2}\lambda}{\sqrt{\log t}}, \qquad \mu_t:=\sqrt{|\lambda_t|^2+1}-1=\sqrt{\frac{2|\lambda|^2}{\log t}+1}-1,
\end{equation*}
it follows that 
\[
\lim_{t\to\infty}\E_{\rho}\left(e^{-\frac{|\lambda|^2}{\log t}A_t} \right)=\lim_{t\to\infty}(\rho)^{\mu_t} \E^{(\mu_t)}_{\rho}\left(\frac{1}{(R(t))^{\mu_t}}\right) = e^{-|\lambda|^2/2},
\]
where the last equality follows from the scaling property of $(R(t))_{t \geq 0}$. 
\end{proof}

\section{Winding of the Brownian motion on $\mathbb{H}P^1$}

\subsection{Winding form on $\mathbb{H}P^1$}

As previously, $\mathbb{H}$ is the quaternionic field and $I,J,K \in \mathbf{SU}(2)$ are the Pauli matrices. Define the quaternionic sphere $\bS^{7}$ by: 
\[
\bS^{7}=\lbrace q=(q_1,q_{2})\in \mathbb{H}^{2}, | q |^2 =1\rbrace.
\]
Then, $\mathbf{SU}(2)$ isometrically acts on $\bS^{7}$ by left multiplication:
\[
q \cdot (q_1,q_2)=(qq_1, qq_2)
\]
and the quotient space $\bS^{7}/ \mathbf{SU}(2)$ is the quaternionic projective line $\bH P^1$. The quaternionic K\"ahler metric on  $\bH P^1$ is such that the projection map  $\bS^{7}\to \bH P^1$ is a Riemannian submersion with totally geodesic fibers isometric to $\mathbf{SU}(2)$. Note that the corresponding  fibration
\[
\mathbf{SU}(2)\to \bS^{7}\to \bH P^1
\]
is called the quaternionic Hopf fibration.  One can parametrize $\bH P^1$ using the quaternionic  inhomogeneous coordinate:
\[
w=q_2^{-1} q_1, \quad q=(q_1,q_2) \in \bS^{7}
\]
with the convention that $0^{-1} q_1=\infty$. This allows to identify $\bH P^1$ with the one-point compactification $\mathbb{H}\cup \{ \infty \}$. This identification will be in force in the sequel. In inhomogeneous coordinate, the Riemannian distance from 0 is given by the formula:
\[
r=\arctan |w|.
\]

If $\gamma: [0,+\infty) \to \mathbb{H}P^1 \setminus \{ 0 , \infty \}$ is a $C^1$-path, one similarly consider its polar decomposition:
\[
\gamma (t) = | \gamma (t) | \Theta (t)
\]
with $\Theta (t) \in \mathbf{SU}(2)$ and define the winding path $\theta (t) \in \mathfrak{su}(2)$ as
\[
\theta(t) = \int_0^t \Theta (s)^{-1}  d\Theta (s).
\]
The quaternionic winding form on $\bH P^1$ is then the $\mathfrak{su}(2)$-valued one-form $\eta$ such that:
\[
\theta(t)=\int_{\gamma[0,t]} \eta=\frac{1}{2}\int_0^t \frac{ \overline{\gamma}(t)d\gamma(t)-d\overline{\gamma}(t)\gamma(t)}{|\gamma(t)|^2}, \quad t \geq 0.
\]

%

\subsection{Asymptotic winding of the Brownian motion on $\mathbb{H}P^1$}

The generator of the Brownian motion $(W(t))_{t \geq 0}$ on $\mathbb{H}P^1$ is half of the Laplacian $\Delta_{\HP^1}$  which is given by (see \cite{VV}, page 75):
\[
\Delta_{\HP^1}=4(1+\rho^2)^2\mathrm{Re}\left( \frac{\partial^2}{\partial \overline{w} \partial w}\right) -8(1+\rho^2)\mathrm{Re}\, \left(w\frac{\partial}{\partial w}\right),
\]
where $\rho:= |w| = \tan(r)$. In real coordinates, we have $w=t+xI+yJ+zK$ and 
\begin{equation*}
\frac{\partial}{\partial w}:=\frac12\left(\frac{\partial}{\partial t}-\frac{\partial}{\partial x}I-\frac{\partial}{\partial y}J-\frac{\partial}{\partial z}K\right).
\end{equation*}
Thus,
\[
\Delta_{\HP^1}=\sec^4r\left(\frac{\partial^2}{\partial t^2}+\frac{\partial^2}{\partial x^2}+\frac{\partial^2}{\partial y^2}+\frac{\partial^2}{\partial z^2}\right)-4\sec^2r\left(t\frac{\partial}{\partial t}+x\frac{\partial}{\partial x}+y\frac{\partial}{\partial y}+z\frac{\partial}{\partial z} \right).
\]
Equivalently, the Brownian motion $(W(t))_{t \geq 0}$ in $\mathbb{H}P^1$ solves the stochastic differential equation: 
\begin{equation*}
dw(t)=\sec^2 r(t)dW(t)-2\sec^2r(t)w(t)dt
\end{equation*}
where $\tan r(t) =\rho (t)= |w(t)|$ and $W$ is a standard Brownian motion in $\mathbb{H}$. Thus, we can write the winding process as
\[
\zeta(t)=\frac12\int_0^t \frac{ \overline{w}(s)dW(s)-d\overline{W}(s)w(s)}{\sin^2r(s)}.
\]
As in the flat setting, the study of $\zeta$ makes use of the following skew-product decomposition.

\begin{lemma}
Let  $w$ be a Brownian motion on $\mathbb{H}P^1$ not started from 0 or $\infty$. There exist a 
  Jacobi process $(r(t))_{t \ge 0}$ with generator
 \[
 \frac{1}{2}\left(\frac{\partial^2}{\partial r^2}+6\cot 2r\frac{\partial}{\partial r}\right)
 \]
  and a Brownian motion $\Theta (t)$ on $\mathbf{SU}(2)$ independent from the process  $(r(t))_{t \ge 0}$ such that 
 \[
 W(t)=\tan r(t) \, \Theta_{A_t},
 \]
 where 
 \begin{equation*}
A_t:= \int_0^t \frac{4 ds}{\sin^2(2r(s))}.
\end{equation*}
\end{lemma}
\begin{proof}
This follows from \cite{PR} and the fact that the operator $(1/2)\Delta_{\HP^1}$ may be decomposed in polar coordinates as:
\begin{equation}\label{Polar}
\frac{1}{2}\left(\frac{\partial^2}{\partial r^2}+6\cot 2r\frac{\partial}{\partial r}+\frac{4}{\sin^2 2r}\Delta_{\mathbf{SU}(2)} \right).
\end{equation}
\end{proof}

As a consequence, we obtain the equality in distribution:
\begin{equation}\label{W-eq-HP}
\zeta (t) \overset{\mathcal{D}}=\beta\left(\int_0^t \frac{4ds}{\sin^2(2r(s))}\right), 
\end{equation}
where $\beta$ is a $3$-dimensional standard Brownian motion which is independent from the process $r$. The analogue of Theorem \ref{T1} for the quaternionic projective line is: 
\begin{theorem}\label{windingHP}
When $t \to \infty$, we have
\[
\frac{\zeta(t)}{t} \to \mathcal{N}(0, 2\mathrm{Id}_3),
\]
in distribution.
\end{theorem}

\begin{proof}
Let $\lambda=(\lambda_1,\lambda_2,\lambda_3) \in \R^3$ and use \eqref{W-eq-HP} to write: 
\[
\mathbb{E}\left( e^{i \lambda\cdot \zeta(t)} \right)=\mathbb{E}\left( e^{-\frac{|\lambda|^2}{2} \int_0^t \frac{4ds}{\sin^2 2r(s)}} \right)=e^{-2|\lambda|^2t} \mathbb{E}\left( e^{-2|\lambda|^2 \int_0^t \cot^2 2r(s) ds} \right).
\]
From \eqref{Polar}, the process $r$ is the (unique) solution of the stochastic differential equation: 
\[
r(t)=r (0)+3\int_0^t \cot 2r(s) ds +\kappa_t, \quad r (0) \in ]0, \pi/2[,
\]
where $\kappa$ is a standard Brownian motion. In order to apply Girsanov's Theorem, we introduce the following local martingale: 
\begin{align*}
D_t^{(\mu)}=\exp \left( 2\mu \int_0^t \cot 2r(s) d\gamma(s) -2\mu^2 \int_0^t \cot^2 2r(s) ds \right), \quad \mu \geq 0.
\end{align*}
From It\^o's formula, we readily derive: 
\[
D_t^{(\mu)} =e^{2\mu  t}\left(\frac{\sin 2r(t)}{\sin 2r (0)}\right)^{\mu} \exp\left(-2(\mu^2+2\mu) \int_0^t \cot^2 2r(s) ds \right),
\]
which shows in particular that $D_t^{(\mu)}$ is a martingale. Now, consider the new probability measure $\p^{(\mu)}$:
\[
\mathbb{P}_{/ \mathcal{F}_t} ^{\mu}=D^\mu_t \mathbb{P}_{/ \mathcal{F}_t}=(\sin 2r (0))^{-\mu} e^{2\mu t} (\sin 2r(t))^{\mu} e^{- 2(\mu^2+2\mu) \int_0^t \cot^2 2r(s)ds} \mathbb{P}_{/ \mathcal{F}_t},
\]
where $(\mathcal{F}_t)_{t \geq 0}$ is the natural filtration of $r$. If we choose $\mu=\sqrt{|\lambda|^2+1}-1$, then we get:
\[
\mathbb{E}\left( e^{i \lambda \cdot\zeta(t)} \right)=(\sin 2r (0))^\mu e^{-2(|\lambda|^2+\mu) t}\E^{(\mu)}\left(\frac{1}{(\sin2r(t)) ^{\mu}}\right).
\]
Moreover, Girsanov's theorem implies that the process
\[
\xi(t) : =\gamma(t)-2\mu\int_0^t \cot 2r(s)ds 
\]
is a Brownian motion under $\p^\mu$. Consequently, 
\[
d r(t)= d\xi(t)+(2\mu+3)\cot 2r(t)dt,
\] 
so that under the probability $\mathbb{P}^\mu$, $r$ is a Jacobi diffusion with generator:
\[
\mathcal{L}^{(\mu+1,\mu+1)}=\frac{1}{2} \frac{\partial^2}{\partial r^2}+\left(\left(\mu+\frac{3}{2}\right)\cot r-\left(\mu+\frac{3}{2}\right) \tan r\right)\frac{\partial}{\partial r}.
\]
Writing,
\[
\mathbb{E}\left( e^{i \lambda\cdot \frac{\zeta(t)}{\sqrt{t}}} \right)=(\sin 2r (0))^{\sqrt{\frac{|\lambda|^2}{t}+1}-1} e^{-2\left(\frac{|\lambda|^2}{t}+\sqrt{\frac{|\lambda|^2}{t}+1}-1\right) t}
\E^{\left(\sqrt{\frac{|\lambda|^2}{t}+1}-1\right)}\left((\sin2r(t)) ^{1-\sqrt{\frac{|\lambda|^2}{t}+1}}  \right) 
\]
 we end up with the limit
 \[
\lim_{t \to \infty} \mathbb{E}\left( e^{i \lambda\cdot \frac{\zeta(t)}{\sqrt{t}}} \right)=\lim_{t \to \infty}e^{-2(\sqrt{|\lambda|^2t+t^2}-t)}= e^{-|\lambda|^2},
\]
as required.
\end{proof}

\begin{remark}
Using the semi-group density of the Jacobi process with equal parameters (see e.g. the appendix of \cite{BW}), we can derive a series representation of the characteristic function of $\zeta$ in the basis of ultraspherical polynomials 
(see \cite{Dem} for the details of computations relative to the complex projective line). 
\end{remark}

\section{Winding of the Brownian motion on $\mathbb{H}H^1$}

\subsection{Winding form on $\mathbb{H}H^1$}

The quaternionic anti-de Sitter space $\mathbf{AdS}^{7}(\mathbb{H})$ is defined as the quaternionic pseudo-hyperboloid:
\[
\mathbf{AdS}^{7}(\mathbb{H})=\lbrace q=(q_1,q_{2})\in \mathbb{H}^{2}, \| q \|^2_H =-1\rbrace,
\]
where 
\[
\|q\|_H^2 :=|q_1|^2-|q_{2}|^2.
\]

The group $\mathbf{SU}(2)$, viewed as the set of unit quaternions, acts isometrically on $\mathbf{AdS}^{7}(\mathbb{H})$ by left multiplication and the quotient space 
\begin{equation*}
\mathbf{AdS}^{7}(\mathbb{H})/ \mathbf{SU}(2)
\end{equation*}
is the quaternionic hyperbolic space $\bH H^1$ endowed with its canonical quaternionic K\"ahler metric. One can parametrize $\bH H^1$ using the quaternionic  inhomogeneous coordinate
\[
w=q_2^{-1} q_1, \quad q=(q_1,q_2) \in \mathbf{AdS}^{7}(\mathbb{H}).
\]
This allows the identification $\bH H^1$ with the unit open ball in $\mathbb{H}$ and will be in force in the sequel. In inhomogeneous coordinates, the Riemannian distance $r$ from 0 is given by the formula
\[
\tanh r= |w|.
\]

If $\gamma: [0,+\infty) \to \mathbb{H}H^1 \setminus \{ 0  \}$ is a $C^1$ path, as before, one can consider its polar decomposition
\[
\gamma (t) = | \gamma (t) | \Theta (t)
\]
with $\Theta (t) \in \mathbf{SU}(2)$ and define the winding path $\theta (t) \in \mathfrak{su}(2)$ as
\[
\theta(t) = \int_0^t \Theta (s)^{-1}  d\Theta (s).
\]
The quaternionic winding form on $\bH H^1$ is then  the $\mathfrak{su}(2)$-valued one-form $\eta$ such that
\[
\theta(t)=\int_{\gamma[0,t]} \eta=\frac{1}{2}\int_0^t \frac{ \overline{\gamma}(t)d\gamma(t)-d\overline{\gamma}(t)\gamma(t)}{|\gamma(t)|^2}, \quad t \geq 0.
\]

\subsection{Asymptotic winding of the Brownian motion on $\mathbb{H}H^1$}

In inhomogeneous coordinates the Laplacian on $\HH^1$ is given by (see \cite{VV} page 45)
\[
\Delta_{\HH^1}=4(1-\rho^2)^2\mathrm{Re} \frac{\partial^2}{\partial \overline{w} \partial w}+8(1-\rho^2)\mathrm{Re}\, w\frac{\partial}{\partial w}. 
\]
In real coordinates we have: 
\[
\Delta_{\HH^1}=\sech^4r\left(\frac{\partial^2}{\partial t^2}+\frac{\partial^2}{\partial x^2}+\frac{\partial^2}{\partial y^2}+\frac{\partial^2}{\partial z^2} \right)+4\sech^2r\left(t\frac{\partial}{\partial t}+x\frac{\partial}{\partial x}+y\frac{\partial}{\partial y}+z\frac{\partial}{\partial z} \right).
\]
Let $w(s)=(t(s),x(s), y(s),z(s))$ be the Brownian motion process generated by $\frac12\Delta_{\HH^1}$, then it solves the SDE 
\begin{equation}\label{eq-w-quater-h}
dw(t)=\sech^2 r(t)dW(t)+2 \sec^2(t)w(t)dt
\end{equation}
where $\tanh r(t)=\rho (t)= |w(t)|^2$ and $W$ is again a standard quaternionic Brownian motion. 

The winding process of a Brownian motion on $\HH^1$ is then given by
\[
\zeta(t)=\mathrm{Im}\int_0^t w^{-1}(s)dw(s)=\frac12\int_0^t \frac{ \overline{w}(s)dw(s)-d\overline{w}(s)w(t)}{|w(s)|^2},
\]
or equivalently, 
\[
\zeta(t)=\frac12\int_0^t \frac{ \overline{w}(s)dW(s)-d\overline{W}(s)w(s)}{\sinh^2r(s)}.
\]

As before, to study $\zeta$, we shall make use of a skew-product decomposition.

\begin{lemma}
Let  $w$ be a Brownian motion on $\mathbb{H}H^1$ not started from 0. There exist a 
  hyperbolic Jacobi process $(r(t))_{t \ge 0}$ with generator
 \[
 \frac{1}{2}\left(\frac{\partial^2}{\partial r^2}+6\coth 2r\frac{\partial}{\partial r}\right)
 \]
  and a Brownian motion $\Theta (t)$ on $\mathbf{SU}(2)$ independent from the process  $(r(t))_{t \ge 0}$ such that 
 \[
 W(t)=\tanh r(t) \, \Theta_{A_t},
 \]
 where 
 \begin{equation*}
A_t:= \int_0^t \frac{4 ds}{\sinh^2(2r(s))}.
\end{equation*}
\end{lemma}
\begin{proof}
This follows from \cite{PR} and the fact that the operator $(1/2)\Delta_{\mathbb{H}H^1}$ may be decomposed in polar coordinates as:
\begin{equation}\label{Polar}
\frac{1}{2}\left(\frac{\partial^2}{\partial r^2}+6\coth 2r\frac{\partial}{\partial r}+\frac{4}{\sinh^2 2r}\Delta_{\mathbf{SU}(2)} \right).
\end{equation}
\end{proof}

%
As a consequence, we obtain the equality in distribution:
\begin{equation}\label{W-eq-HH}
\zeta (t) \overset{\mathcal{D}}=\beta\left(\int_0^t \frac{4ds}{\sinh^2(2r(s))}\right), 
\end{equation}
where $\beta$ is a $3$-dimensional standard Brownian motion which is independent from the process $r$. 

Unlike the  $\mathbb{H}P^1$ case, the Brownian motion on $\mathbb{H}H^1$ is transient and as shown below, the corresponding winding process will have a limit in distribution when $t \to +\infty$. Moreover, the computations of the limiting distribution are more involved compared to the flat and the spherical settings. 
\begin{theorem}\label{windingCP}
For any $\lambda \in \mathbb{R}^3$,  
\[
\lim_{t \rightarrow +\infty} \mathbb{E}[e^{i\lambda \cdot \zeta(t)}] = \tanh (r (0))^{\sqrt{|\lambda|^2+1}-1} \left( 1 +\frac{1}{2\cosh(r (0))^2}(\sqrt{|\lambda|^2+1}-1) \right).
\] 
\end{theorem}

\begin{proof}
Let $\lambda \in \R^3$. We have 
\[
\mathbb{E}\left( e^{i \lambda\cdot \zeta(t)} \right)=\mathbb{E}\left( e^{-2|\lambda|^2 \int_0^t \frac{ds}{\sinh^2 2r(s)}} \right).
\]
The process $r$ solves the stochastic differential equation:
\[
r(t)=r (0)+3\int_0^t \coth(2r(s)) ds +\psi(t),
\]
where $\psi$ is a standard one-dimensional Brownian motion. In order to compute the characteristic function of $\zeta$, we shall look for an exponential local martingale of the form 
\begin{equation*}
D_t^{(\nu, \kappa)} := \exp\{\int_0^t[\nu \coth(r(s)) + \kappa \tanh(r(s))] d\gamma(s) - \frac{1}{2}\int_0^t [\nu \coth(r(s)) + \kappa \tanh(r(s))]^2ds \}. 
\end{equation*}
To this end, we use It\^o's formula to derive: 
\begin{equation*}
\left(\frac{\sinh(r(t))}{\sinh(r (0))}\right)^{\nu} =  \exp\{\nu \int_0^t\coth(r(s)) d\gamma(s) + \nu \int_0^t\coth^2(r(s))ds + 2\nu t\}, \quad r (0) > 0,
\end{equation*}
 \begin{equation*}
\left(\frac{\cosh(r(t))}{\cosh(r (0))}\right)^{\kappa} =  \exp\{\kappa \int_0^t\tanh(r(s)) d\gamma(s) + \kappa \int_0^t\tanh^2(r(s))ds + 2 \kappa t\},
\end{equation*}
so that 
\begin{multline*}
D_t^{(\nu, \kappa)} := e^{-2(\nu +\kappa+\nu\kappa/2)t} \left(\frac{\sinh(r(t))}{\sinh(r (0))}\right)^{\nu}\left(\frac{\cosh(r(t))}{\cosh(r (0))}\right)^{\kappa} 
\\\exp\{-\frac{\nu^2+2\nu}{2}\int_0^t \coth^2(r(s))ds  -\frac{\kappa^2 +2\kappa}{2} \int_0^t \tanh(r(s))^2ds \}. 
\end{multline*}
Writing 
\begin{align*}
\frac{1}{\sinh^2 2r(s)} = \frac{[\coth(r(s)) - \tanh(r(s))]^2}{4} = \frac{\coth^2(r(s)) + \tanh^2(r(s)) - 2}{4}, 
\end{align*}
and choosing  
\begin{equation*}
\nu = \sqrt{1+|\lambda|^2} - 1 = -\kappa -2, 
\end{equation*}
then 
\begin{equation*}
D_t^{(\nu, \kappa)} := e^{4t} \left(\frac{\tanh(r(t))}{\tanh(r (0))}\right)^{ \sqrt{1+|\lambda|^2} - 1}\left(\frac{\cosh(r (0))}{\cosh(r(t))}\right)^{2}e^{-2|\lambda|^2 \int_0^t \frac{ds}{\sinh^2 2r(s)}}.
\end{equation*}
Since $\nu \geq 0$ then $D_t^{(\nu, \kappa)}$ is a bounded local martingale.
So $D_t^{(\nu, \kappa)}$ is a martingale and we can define the probability measure 
\begin{equation*}
\mathbb{P}_{/ \mathcal{F}_t}^{(\nu, \kappa)} := D_t^{(\nu, \kappa)} \mathbb{P}_{/ \mathcal{F}_t} 
\end{equation*}
under which the process $(r(t))_{t \geq 0}$ solves the SDE 
\begin{equation*}
r(t)=r (0)+\int_0^t \left[\left(\frac{3}{2}+\nu\right)\coth r(s) - \left(\frac{1}{2} + \nu\right) \tanh(r(s))\right]ds +\tilde{\gamma}(t)
\end{equation*}
for some $\mathbb{P}^{(\nu, \kappa)}$-Brownian motion $\tilde{\gamma}$. This is a hyperbolic Jacobi process of parameters ($1+\nu, -1-\nu$) and 
\begin{align*}
\mathbb{E}\left(e^{-2|\lambda|^2 \int_0^t \frac{ds}{\sinh^2 2r(s)}}\right) & = e^{-4t} \mathbb{E}^{(\nu,-\nu-2)}\left\{\left(\frac{\tanh(r(t))}{\tanh(r (0))}\right)^{1-\sqrt{1+|\lambda|^2}}\left(\frac{\cosh(r(t))}{\cosh(r (0))}\right)^{2}\right\}
\\& = e^{-4t} \frac{\left(\tanh(r (0))\right)^{\sqrt{1+|\lambda|^2}-1}}{\cosh^2(r (0))}\mathbb{E}^{(\nu,-\nu-2)}\left\{\left(1-\frac{1}{\cosh^2(r(t))}\right)^{-\nu/2} \cosh^2(r(t))\right).
\end{align*}
Using the generalized binomial Theorem, we get further: 
\begin{equation*}
\mathbb{E}\left(e^{-2|\lambda|^2 \int_0^t \frac{ds}{\sinh^2 2r(s)}}\right)  = e^{-4t}  \frac{\left(\tanh(r (0))\right)^{\sqrt{1+|\lambda|^2}-1}}{\cosh^2(r (0))}\sum_{k \geq 0}\frac{(\nu/2)_k}{k!} \mathbb{E}^{(\nu,-\nu-2)}\left(\frac{1}{\cosh^{2k-2}(r(t))}\right).
\end{equation*}
Consequently, the long-time behavior of $\zeta(t)$ is given by the lowest-order term $k=0$ in the series above: 
\begin{equation}\label{Limit}
\lim_{t \rightarrow +\infty} \mathbb{E}\left(e^{-2|\lambda|^2 \int_0^{+\infty} \frac{ds}{\sinh^2 2r(s)}}\right) = \frac{\left(\tanh(r (0))\right)^{\sqrt{1+|\lambda|^2}-1}}{\cosh^2(r (0))}\lim_{t \rightarrow +\infty}e^{-4t}  \mathbb{E}^{(\nu,-\nu-2)}[\cosh^2(r(t))]. 
\end{equation}
To this end, we need the following lemma: 

\begin{lemma}\label{limit cosh}
For any $\alpha,\beta$, consider the hyperbolic Jacobi process solution of
\[
dr_{\alpha,\beta} (t)=(\alpha \coth r_{\alpha,\beta} (t) +\beta \tanh r_{\alpha,\beta} (t))dt +d\gamma(t), \quad r_{\alpha,\beta} (0)=r (0)>0.
\]
Then
\begin{align*}
\mathbb{E} ((\cosh r_{\alpha,\beta}(t)^2))=\frac{1+2\beta}{2(1+\alpha+\beta)} + e^{2(1+\alpha+\beta)t} \left( \cosh(r (0))^2 -\frac{1+2\beta}{2(1+\alpha+\beta)}\right).
\end{align*}
\end{lemma}

\begin{proof}
Let
\[
f(r)=(\cosh (r))^2.
\]
We have
\[
Lf =2(1+\alpha+\beta) f -(1+2\beta)
\]
where
\[
L=\frac{1}{2} \frac{d^2}{dr^2} +(\alpha \coth r +\beta \tanh r)\frac{d}{dr}.
\]
Thus, denoting
\[
\phi (t) := \mathbb{E} ((\cosh r_{\alpha,\beta}(t)^2)),
\]
and using It\^o's formula, we obtain the following differential equation:
\[
\phi'(t)=2(1+\alpha+\beta) \phi(t) -(1+2\beta),
\]
which proves the Lemma. 
\end{proof}
Specializing the lemma to $\alpha = 3/2+\nu, \beta = -1/2-\nu$, then 
\begin{equation*}
\mathbb{E}^{(\nu,-\nu-2)}[\cosh^2(r(t))] = \lim_{t \rightarrow +\infty}e^{-4t} \mathbb{E}^{(\nu,-\nu-2)}[\cosh^2(r(t))] = \cosh^2(r (0)) + \frac{\nu}{2}. 
\end{equation*}
Plugging this expression into \eqref{Limit}, we are done. 
\end{proof}
We close the paper with an explicit expression of the density of $\zeta_{\infty}$. To this end, we write 
\begin{multline*}
\tanh (r (0))^{\sqrt{|\lambda|^2+1}-1}\left( 1 +\frac{1}{2\cosh(r (0))^2}(\sqrt{|\lambda|^2+1}-1) \right) = \tanh (r (0))^{\sqrt{|\lambda|^2+1}-1} + \\ \frac{\tanh(r (0))}{2}\partial_{r (0)} \tanh (r (0))^{\sqrt{|\lambda|^2+1}-1},
\end{multline*}
and recall from \cite{BMR09} the characteristic function of the three-dimensional relativistic Cauchy random variable: 
\begin{equation*}
e^{-y(\sqrt{|\lambda|^2+1}-1)} = \frac{ye^y}{2\pi^2}\int_{\mathbb{R}^3} e^{i\lambda \cdot x} \frac{K_2(\sqrt{|x|^2+y^2})}{|x|^2+y^2} dx,
\end{equation*}
where $K_2$ is the modified Bessel function of the second kind. After straightforward computations, we end up with:
\begin{corollary}
Let $r (0) > 0$ be the hyperbolic distance from the origin of the Brownian motion in $\mathbb{H}H^1$. Then, the distribution of $\zeta_{\infty}$ is absolutely continuous with respect to Lebesgue measure in $\mathbb{R}^3$ and its density is given by: 
\begin{multline*}
\frac{-\ln(\tanh(r (0)))}{2\pi^2\tanh(r (0))} \frac{K_2(\sqrt{|x|^2+\ln^2(\tanh(r (0)))})}{|x|^2+\ln^2(\tanh(r (0)))} + \\ 
\frac{\tanh(r (0))}{2}\partial_{u}\left\{\frac{-\ln(\tanh(u))}{2\pi^2\tanh(u)} \frac{K_2(\sqrt{|x|^2+\ln^2(\tanh(u))})}{|x|^2+\ln^2(\tanh(u))}\right\}(r (0)).
\end{multline*}
\end{corollary}

\end{document}